\documentclass[12 pt]{amsart}

\usepackage{amssymb,latexsym,amsmath,epsfig,amsthm} %% Add other packages as necessary
\usepackage{graphicx}
\usepackage[normalem]{ulem}
\usepackage{marvosym}
\usepackage{verbatim}
\usepackage{enumitem}
\usepackage{thmtools}
\usepackage{multicol}

\newtheorem{theorem}{Theorem}
\newtheorem{lemma}{Lemma}

\newtheorem{corollary}{Corollary}
\newtheorem{remark}{Remark}

\font\smallit=cmti10

\begin{document}
\begin{center}
\uppercase{\bf Basic enumeration of graph compositions with a restricted number of components}
\vskip 20pt
{\bf Todd Tichenor}\\
{\smallit International College of Beijing, China Agricultural University, Beijing 100083, P.R. China}\\
{\tt todd\textunderscore tichenor\MVAt cau.edu.cn}\\
\vskip 10pt
\end{center}
\vskip 30pt
%\centerline{\smallit Received: , Revised: , Accepted: , Published: } % We will fill in the dates
\vskip 30pt
\centerline{\bf Abstract}
\noindent
The concept of graph compositions is related to several number theoretic concepts, including partitions of positive integers and the cardinality of the power set of finite sets. This paper examines graph compositions where the total number of components is restricted and illustrates a connection between graph compositions and Stirling numbers of the second kind.
\pagestyle{myheadings}
%\markright{\smalltt INTEGERS: 14 (2014)\hfill}
\thispagestyle{empty}
\baselineskip=12.875pt
\vskip 30pt
\section{Introduction}
The idea of graph compositions was introduced by Knopfmacher and Mays \cite{km}. For a graph $G$, a graph composition is defined as a partition of $V(G)$ where each member of the partition induces a connected subgraph (i.e. every composition of $G$ induces a unique subgraph of $G$ with connected components). The composition number of $G$ [denoted $C(G)$] is the number of graph compositions of $G$. For example, a graph composition of the path on $n$ vertices corresponds to a composition of the integer $n$; additionally, a graph composition of the complete graph on $n$ vertices corresponds to a general partition of a set of cardinality $n$, therefore implying that its composition number is given by the $n\textsuperscript{th}$ Bell number \cite{km}.
\\
\\
Previous work has been conducted on the composition number of $K_n^{-G}$ (referred to in \cite{TM} as the deletion of $G$ from $K_n$), formed by deleting the edges of $G$ from the complete graph on $n$ vertices. Specifically, \cite{TM} connects $C(K_n^{-G})$ to sequences of the Bell numbers. This paper is comprised of two parts: the first examines the composition number of a few basic graphs where the number of components of each graph composition is restricted; the second studies compositions of the same kind in $K_N^{-G}$ and relates them to Stirling numbers of the second kind.
\\
\\
For a graph $G$ with $|V(G)| = n$, $\mathcal{C}^k(G)$ will denote the set of graph compositions of $G$ with exactly $k$ components and $C^k(G)$ will denote the number of compositions of $G$ with exactly $k$ components. For the theorems that follow, it is to be understood that $0 < k \leq n.$
\section{Restricted graph compositions of basic graphs}
\begin{lemma}\label{lemma1}
$C(G) = \underset{k=1}{\overset{n}{\sum}}C^k(G)$ for any graph $G$ such that $|V(G)| = n.$
\end{lemma}
\begin{remark}\label{rem10}
The above lemma may seem trivial, but it is none-the-less useful for verifying $C^k(G)$ once it is found.
It should also be noted that $C^1(G) = 1$, $C^n(G) = 1$, and $C^k(G) = 0$ for $k > n$.
\end{remark}

\begin{theorem}\label{theorem28}
If $T_n$ is a tree with $n$ vertices, then $C^k(T_n) = \binom{n-1}{k-1}$.
\end{theorem}

\begin{proof}
It is trivial to notice that the deletion of a single edge from any subgraph of $T_n$ will result in a distinct subgraph which has exactly 1 more component than its original subgraph. Since $T_n$ is connected, we must delete $k-1$ edges to produce a subgraph of $T_n$ which has exactly $k$ components. Hence $C^k(T_n) = \binom{n-1}{k-1}$.
\end{proof}

\begin{theorem}\label{theorem30}
If $G = G_1 \cup G_2$ where $G_1$ and $G_2$ are disjoint, then $C^k(G) = \underset{j=1}{\overset{k-1}{\sum}}\left[C^j(G_1) \cdot C^{k-j}(G_2)\right]$.
\end{theorem}

\begin{proof}
 Let $k \geq j$ and $C$ be a composition of $G_1$ which has $j$ components. There are exactly $C^{k-j}(G_2)$ compositions of $G$ which contain all components of $C$ and have exactly $k$ components; furthermore, $C$ can be chosen in exactly $C^j(G)$ ways. Hence, $C^k(G) = \underset{j=1}{\overset{k-1}{\sum}}\left[C^j(G_1)\cdot C^{k-j}(G_2)\right]$.
\end{proof}

\begin{corollary}\label{cor7}
Let $G = G_1 \cup G_2$. If $G_1$ and $G_2$ share exactly one vertex, then $C^k(G) = \underset{j=1}{\overset{k}{\sum}}\left[C^j(G_1)\cdot C^{k+1-j}(G_2)\right].$
\end{corollary}

\begin{proof}
Let $v$ be the vertex which $V(G_1)$ and $V(G_2)$ share and $C \in \mathcal{C}^k(G)$. If we separate $C$ into 2 disjoint graphs by disconnecting $v$ from all vertices in $V(G_2)$ and adding a new vertex which is adjacent to all of the original neighbors of $v$ in $V(G_2)$, then the result is a composition of $k+1$ components. So $C^k(G)$ is the same as the number of compositions of the disjoint union of $G_1$ and $G_2$ with exactly $k+1$ components; hence, by Theorem \ref{theorem30}, $C^k(G) = \underset{j=1}{\overset{k}{\sum}}\left[C^j(G_1)\cdot C^{k+1-j}(G_2)\right]$.
\end{proof}
~\\
~\\
For the following, if $G$ is a graph and $H$ a subgraph of $G$, then $G^{-H}$ will represent the graph where $V(G^{-H}) = V(G)$ and $E(G^{-H}) = E(G) \backslash E(H)$; additionally, if $C$ is a composition of $G$ and $G_C$ represents the subgraph of $G$ induced by $C$, then $C^{-H}$ will represent the composition
of $G$ which induces $\left(G_{C}\right)^{-H}$.
\begin{theorem}\label{theorem41}
Let $G = G_1 \cup G_2$ and $V(G_1) \cap V(G_2) = \emptyset$. If $G$ has exactly one edge incident to vertices from $G_1$ and $G_2$, then $C^k(G) = \underset{j=1}{\overset{k-1}{\sum}}C^j(G_1) \cdot \left[C^{k+1-j}(G_2) + C^{k-j}(G_2)\right] + C^{k}(G_1)$.
\end{theorem}

\begin{proof}
Let $e$ be the edge incident to vertices in $G_1$ and $G_2$ and $C \in \mathcal{C}^k(G).$ If the vertices of $e$ in $C$ are connected, then $C^{-\{e\}}$ will have exactly $k+1$ components. Hence, there are exactly $C^{k+1}(G^{-\{e\}})$ compositions of this form. Also, there are $C^k(G^{-\{e\}})$ compositions of $G^{-\{e\}}$ in which the vertices of $e$ are not connected. Theorem \ref{theorem30} yields the results
\begin{align*}
& C^{k+1}(G^{-\{e\}}) = \underset{j=1}{\overset{k}{\sum}}\left[C^j(G_1) \cdot C^{k+1-j}(G_2)\right]
\\
& C^{k}(G^{-\{e\}}) = \underset{j=1}{\overset{k-1}{\sum}}\left[C^j(G_1) \cdot C^{k-j}(G_2)\right].
\end{align*}
\begin{align*}
\text{Hence, }C^k(G) &= C^{k+1}(G^{-\{e\}}) + C^{k}(G^{-\{e\}})
\\
&= \underset{j=1}{\overset{k}{\sum}}\left[C^j(G_1) \cdot C^{k+1-j}(G_2)\right] + \underset{j=1}{\overset{k-1}{\sum}}\left[C^j(G_1) \cdot C^{k-j}(G_2)\right]
\\
&= \underset{j=1}{\overset{k-1}{\sum}}C^j(G_1) \cdot \left[C^{k+1-j}(G_2) + C^{k-j}(G_2)\right] + C^{k}(G_1)\cdot C^1(G_2)
\\
&= \underset{j=1}{\overset{k-1}{\sum}}C^j(G_1) \cdot \left[C^{k+1-j}(G_2) + C^{k-j}(G_2)\right] + C^{k}(G_1).
\end{align*}
\end{proof}
\begin{comment}
For the following, if $e \in E(K_n)$, then $K_n^-$ represents the graph $K_n^{-\{e\}}$.
\begin{theorem}\label{theorem31}
$C^k(K_n^-)= S(n,k) - S(n-2,k-1)$.
\end{theorem}

\begin{proof}
 If $C$ is a composition of $K_n$ that has exactly $k$ components, then $C$ will not be a composition of $K_n^{-\{e\}}$ if and only if $e$ is a singleton component of $C$. There are exactly $S(n-2,k-1)$  compositions of this form. Hence, Theorem \ref{theorem31} is true.
\end{proof}
\end{comment}
\begin{theorem}\label{theorem32}
$C^k(C_n) = \begin{cases}
          1, & k=1\\
          \binom{n}{k}, & 1 < k \leq n
\end{cases}$

\end{theorem}

\begin{proof}
$k=1$ is the trivial case of our theorem and is discussed in Remark \ref{rem10}. Assume that $k > 1$.
Deleting a single edge from $C_n$ yields $P_n$. Theorem \ref{theorem28} yields $C^k(P_n) = \binom{n-1}{k-1}$.
Since there are exactly $n$ ways of choosing the edge deleted, we have $C^k(C_n) = \binom{n}{k}$.
\end{proof}

\section{Connection to Stirling numbers of the second kind}
Recall that the $(n,k)\textsuperscript{th}$ entry in the array of Stirling numbers of the second kind [denoted by $S(n,k)$] represents the number of partitions of a set $S$ of cardinality $n$, where every partition has exactly $k$ non-empty subsets of $S$. A portion of the array appears below.
\\
\\
\begin{center}
\begin{tabular}{c c c c c c c c c c c}
1 &  &  &  &  &  &  &  &  & \\\\
1 & 1 &  &  &  &  &  &  &  &  \\\\
1 & 3 & 1 &  &  &  &  &  &  &  \\\\
1 & 7 & 6 & 1 &  &  &  &  &  &  \\\\
1 & 15 & 25 & 10 & 1 &  &  &  &  &  \\\\
\end{tabular}
\end{center}
\begin{lemma}\label{lemma2}
$C^k(K_n) = S(n,k)$.
\end{lemma}
\begin{proof}
A quick examination of the definition of $S(n,k)$ above shows that Lemma \ref{lemma2} holds.
\end{proof}
\begin{remark}\label{rem9}
Theorem \ref{theorem28} and Lemma \ref{lemma2} are extreme cases for any connected graph $G$ such that $|V(G)|  = n$.  Hence, it is easily observed that $\binom{n-1}{k-1} \leq C^k(G) \leq S(n,k)$ for any connected $G$ such that $|V(G)| = n$.
\end{remark}

For all that follows, a component $G_i$ of a graph composition of $K_n^{-G}$ will be referred to as \emph{bad} if $V(G_i) \subseteq V(G)$ and the complement of $G$ restricted to $G_i$ is disconnected; otherwise, it will be referred to as a \emph{good} component of $K_n^{-G}$.
\begin{theorem}\label{theorem34}
Let $G$ be a graph with $n \leq N$ vertices.  If $b_{j,m,n}$ represents the number of ways of choosing $m$ disjoint bad components of $K_N^{-G}$ such that the cardinality of the union of vertices of all bad components is $j$, then $C^k(K_N^{-G}) = \underset{j=0}{\overset{n}{\sum}}\underset{m=0}{\overset{j}{\sum}}(-1)^mb_{j,m,n}\cdot S(N-j,k-m)$.
\end{theorem}

\begin{proof}
\noindent We begin by denoting the set of all compositions of $K_N$ with exactly $k$ components by $\Gamma$. If $B(\gamma)$ denotes
the number of disjoint bad components contained in $\gamma$ for all $\gamma \in \Gamma$, then
\[
C^k(K_N^{-G}) = \underset{B(\gamma) = 0}{\underset{\gamma \in \Gamma}{\sum}}1 =  \underset{B(\gamma) = 0}{\underset{\gamma \in \Gamma}{\sum}}1 +  \underset{B(\gamma) \neq 0}{\underset{\gamma \in \Gamma}{\sum}}0 =  \underset{\gamma \in \Gamma}{\sum}\underset{m=0}{\overset{B(\gamma)}{\sum}}(-1)^m\binom{B(\gamma)}{m}
\]
since the alternating sum of the $m^{th}$ row of Pascal's Triangle is $0$ for $B(\lambda) > 0$ and $1$ for $B(\lambda)=0$. Note that $\binom{B(\gamma)}{m}$ will count the number of times $\gamma$ is counted as a composition of $K_N$ with at least $m$ disjoint \emph{bad} components for a fixed $\gamma$ and $m$.
\\
$$\underset{\gamma \in \Gamma}{\sum}\underset{m=0}{\overset{B(\gamma)}{\sum}}\binom{B(\gamma)}{m} = \underset{j=0}{\overset{n}{\sum}}\underset{m=0}{\overset{n}{\sum}}b_{j,m,n}S(N-j, k-m).$$
\\
So
\begin{align*}
C^k(K_N^{-G}) &= \underset{\gamma \in \Gamma}{\sum}\underset{m=0}{\overset{B(\gamma)}{\sum}}(-1)^m \binom{B(\gamma)}{m}
\\
&= \underset{j=0}{\overset{n}{\sum}}\underset{m=0}{\overset{n}{\sum}}(-1)^mb_{j,m,n}S(N-j, k-m).
\end{align*}
\end{proof}
\begin{remark}\label{rem11}
It should be noted that:
\begin{enumerate}[label = (\roman*)]
\item $b_{j,m,n} = 0$ if $j > n$, $m > j$, $m > n$, or $j > m = 0$ (restrictions set by the number of vertices present in $G$ and the number of components being chosen).
\\
\item $b_{j,m,n} = 1$ if $j = m = 0$ (since there's only one way to choose no bad components (or vertices) of $G$).
\\
\item $b_{1,m,n} = 0$ (since there's no way to choose bad components from $G$ with only one vertex).
\end{enumerate}
\end{remark}
\section{Deletions of specific graphs From complete graphs}
\begin{theorem}\label{theorem35}
If $p_{j,m,n}$ represents the number of ways of choosing $m$ disjoint bad components of $K_N^{-P_n}$ such that the cardinality of the union of vertices of all components is $j$, then $C^k(K_N^{-P_n})= \underset{j=0}{\overset{n}{\sum}}\underset{m=0}{\overset{j}{\sum}}(-1)^mp_{j,m,n}\cdot S(N-j,k-m)$ and $p_{j,m,n} = p_{j,m,n-1} + p_{j-2,m-1,n-2} + p_{j-3,m-1,n-3}$ for $j \geq 3$, $m \geq 1$, and $n \geq 3$.
\end{theorem}

\begin{proof}
$C^k(K_N^{-P_n}) = \underset{j=0}{\overset{n}{\sum}}\underset{m=0}{\overset{j}{\sum}}(-1)^mp_{j,m,n}\cdot S(N-j,k-m)$ is a result of application of Theorem \ref{theorem34} to $K_N^{-P_n}$. If $C$ is a composition of $K_N$, then $C$ will \textbf{not} be a
composition of $K_N^{-P_n}$ if and only if $C$ contains a bad component of $K_N^{-P_n}$. It is noted in \cite{TM} that the deletion of any subpath of $P_n$ of length $t \geq 3$ from $K_N$ yields a connected component; hence, all bad components of $K_N^{-P_n}$ are either single edges or subpaths of length 2.
\\
\\
\indent If $u$ represents one of the terminal vertices of $P_n$ and $\mathcal{C}$ is a component which contains $u$, then one of three cases must occur:
\\
\begin{enumerate}[label=(\roman*)]
\item $\mathcal{C}$ is not a bad component. There are $p_{j,m,n-1}$ ways of choosing $m$ disjoint bad components from $V(P_n)$ such that the cardinality of the union of vertices of all components which include $\mathcal{C}$ is $j$.
    \\
\item $\mathcal{C}$ is a single edge bad component. There are $p_{j-2,m-1,n-2}$ ways of choosing $m$ disjoint bad components from $V(P_n)$  such that the cardinality of union of vertices of components which include $\mathcal{C}$ is $j$.
    \\
\item $\mathcal{C}$ is a 3 element bad component. There are $p_{j-3,m-1,n-3}$ ways of choosing $m$ disjoint bad components from $V(P_n)$  such that the cardinality of union of vertices of components which include $\mathcal{C}$ is $j$.
\end{enumerate}
The above ``world encompassing" cases yield the result
$$p_{j,m,n} = p_{j,m,n-1} + p_{j-2,m-1,n-2} + p_{j-3,m-1,n-3}.$$
\end{proof}

\begin{theorem}\label{theorem39}
If $F^k(x,y,z) = \underset{j=0}{\overset{\infty}{\sum}}\underset{m=0}{\overset{\infty}{\sum}}\underset{n=0}{\overset{\infty}{\sum}}(-1)^mp_{j,m,n}x^ny^mz^j$,
\\
then $F^k(x,y,z) = \frac{1}{1 - x + x^2yz^2 + x^3yz^3}$.
\end{theorem}
\begin{proof}
$F^k(x,y,z) = \underset{j=0}{\overset{\infty}{\sum}}\underset{m=0}{\overset{\infty}{\sum}}\underset{n=0}{\overset{\infty}{\sum}}(-1)^mp_{j,m,n}x^ny^mz^j
 = \underset{j=0}{\overset{\infty}{\sum}}\underset{m=0}{\overset{\infty}{\sum}}\underset{n=0}{\overset{\infty}{\sum}}p_{j,m,n}x^n(-y)^mz^j$.
 \\
 \\
As stated in Theorem \ref{theorem35}, $p_{j,m,n}$ is the number of ways of choosing $m$ disjoint bad components of $K_N^{-P_n}$ where the cardinality of bad vertices is $j$. Every component of a composition of $K_N^{-P_n}$ with exactly $m$ bad components is either:
\\
\begin{enumerate}[label= (\roman*)]
\item A good component - represented by $x$.
\\
\item A 2 element bad component- represented by $x^2(-y)z^2.$
\\
\item A 3 element bad component- represented by $x^3(-y)z^3.$
\end{enumerate}
~\\
~\\
$(x - x^2yz^2 - x^3yz^3)^m$ represents the number of ways one can choose at most $m$ disjoint bad components of a composition of $K_N^{-P_n}$. $F^k(x,y,z)$ contains coefficients for all $m$, so

\begin{align*}
F^k(x,y,z) &= \underset{m=0}{\overset{\infty}{\sum}}(x - (x^2yz^2 + x^3yz^3))^m
\\
&= \frac{1}{1 - x+x^2yz^2 + x^3yz^3}.
\end{align*}
\end{proof}

\begin{theorem}\label{theorem36}
If $p_{j,m,n}$ is defined as in Theorem \ref{theorem35} and $c_{j,m,n}$ represents the number of ways of choosing $m$ disjoint \textit{bad} components of $K_N^{-C_n}$ where the number of vertices of the \textit{bad} components is $j$, then $C^k(K_N^{-C_n}) = \underset{j=0}{\overset{n}{\sum}}\underset{m=0}{\overset{j}{\sum}}(-1)^mc_{j,m,n}\cdot S(N-j,k-m)$ and $c_{j,m,n} = p_{j,m,n-1} + 2 \cdot p_{j-2,m-1,n-2} + 3 \cdot p_{j-3,m-1,n-3}$ for $j$ and $n \geq 3$, and $n \not =j$ for $n \in \{3,4\}$.
\end{theorem}
\begin{proof}$C^k(K_N^{-C_n}) = \underset{j=0}{\overset{n}{\sum}}\underset{m=0}{\overset{j}{\sum}}(-1)^mc_{j,m,n}\cdot S(N-j,k-m)$ follows from the application of Theorem \ref{theorem34} to $K_N^{-C_n}$. If $C$ is a composition of $K_N$, then $C$ will \textbf{not} be a composition of $K_N^{-C_n}$ if and only if $C$ contains a bad component of $K_N^{-C_n}$. It is noted in \cite{TM} that $K_N^{-C_n}$ will have bad components that are either single edges, subpaths of length 2, or cycles of length 3 or 4.
\\
~\\
\indent If $w \in V(C_n)$ and $\mathcal{C}$ is a component that contains $w$, then one of four cases occurs:
\\
\begin{enumerate}[label=(\roman*)]
 \item $\mathcal{C}$ is not a bad component. There are $p_{j,m,n-1}$ ways of choosing $m$ disjoint bad components from $V(C_n)$ (with cardinality of the union of vertices $j$) which do not include $\mathcal{C}$.
     \\
\item $\mathcal{C}$ is a 2 element bad component. There are 2$\cdot p_{j-2,m-1,n-2}$ ways of choosing $m$ disjoint bad components from $V(C_n)$ (with cardinality of the union of vertices $j$) which include $\mathcal{C}$.
    \\
\item $\mathcal{C}$ is a 3 element bad component. If $\mathcal{C}$ is a path ($n > 3$), then there are 3$\cdot p_{j-3,m-1,n-3}$ ways of choosing $m$ disjoint bad components from $V(C_n)$ (with cardinality of the union of vertices $j$) which include $\mathcal{C}$. If $\mathcal{C}$ is a cycle ($n = 3$), then there is exactly one way of choosing 1 bad component.
    \\
\item $\mathcal{C}$ is a 4 element bad component. There is exactly one way of choosing 1 bad component.
\end{enumerate}
~\\
\indent The above ``world encompassing" cases give us
\begin{center}
$c_{j,m,n} = p_{j,m,n-1} + 2\cdot p_{j-2,m-1,n-2} + 3\cdot p_{j-3,m-1,n-3}$.
\end{center}
for $j$ and $n \geq 3$ and $n \not=j \in\{3,4\}$.
\end{proof}

\begin{theorem}\label{theorem40}
If $G^k(x,y,z) = \underset{j=0}{\overset{\infty}{\sum}}\underset{m=0}{\overset{\infty}{\sum}}\underset{n=0}{\overset{\infty}{\sum}}(-1)^mc_{j,m,n}x^ny^mz^j$, then $G^k(x,y,z) =1 + x^2yz^2 + 2x^3yz^3 - x^4yz^4 + \frac{x - 2x^2yz^2 - 3x^3yz^3}{1-x+x^2yz^2+x^3yz^3}$.
\end{theorem}

\begin{proof}
We begin by calculating $G^k(x,y,z)$ for $n\not=j\in\{0,2,3,4\}$ and calling it $H^k(x,y,z)$. As in Theorem \ref{theorem36}, $c_{j,m,n}$ represents the number of ways of choosing $m$ disjoint bad components of $K_N^{-C_n}$ where the cardinality
of the set of bad vertices is $j$. Every component of a composition of $K_N^{-C_n}$ with exactly $m$ bad components is either:
\\
\begin{enumerate}[label = (\roman*)]
\item A good component represented by $x$. There are $p_{j,m,n-1}$ ways of choosing $m$ bad components if $n \geq 1$ (i.e. if $n > j=0$). 
\\
\item Contained in a 2 element bad component represented by $x^2(-y)z^2$. There are $2\cdot p_{j-2,m-1,n-2}$ ways of choosing the remaining $m-1$ bad components if $n \geq 3$ (i.e. if $n\not=j=2$). Otherwise, there is exactly 1 way.
    \\
\item Contained in a 3 element bad component represented by $x^3(-y)z^3$. There are $3\cdot p_{j-3,m-1,n-3}$ ways of choosing the remaining $m-1$ bad components  if $n \geq 4$ (i.e. $n\not=j=3$). Otherwise, there is exactly 1 way.
    \\
\item Contained in a 4 element bad component represented by $x^4(-y)z^4$. There is exactly one way of choosing such a component. 
\end{enumerate}
~\\
This yields the product $(x-2x^2yz^2-3x^3yz^3)(x-x^2yz^2-x^3yz^3)^m$ which represents the number of ways possible to choose at most $m+1$ bad components
of $K_N^{-C_n}$ for $n\not=j \in\{0,2,3,4\}$. Summing the expression over all $m \geq 0$ yields
\\
\\
$H^k(x,y,z) = \underset{m=0}{\overset{\infty}{\sum}}(x-2x^2yz^2)(x-x^2yz^2-x^3yz^3)^m = \frac{x-2x^2yz^2-3x^3yz^3}{1-x+x^2yz^2+x^3yz^3}$.
 \\
 \\
 \\If $h_{j,m,n}$ is defined to be the coefficient of $x^ny^mz^j$ in $H^k(x,y,z)$, then
\\
\\
$h_{j,m,n} = \begin{cases}
0&, \text{ for } n=j=0 \text{ and } m=0\\
-2&, \text{ for } n=j=2 \text{ and } m=1\\
-3&, \text{ for } n=j=3 \text{ and } m=1\\
0&, \text{ for } n=j=4 \text{ and } m=1\\
(-1)^mc_{j,m,n}&, \text{ otherwise}
\end{cases}$
\\
\\
and
\\
\\
$(-1)^mc_{j,m,n} = \begin{cases}
1&, \text{ for } n=j=0 \text{ and } m=1\\
-1&, \text{ for } n=j\in\{2,3,4\} \text{ and } m=1
\end{cases}$
~\\
~\\
~\\
Hence, $G^k(x,y,z)$
\begin{align*}
&\begin{aligned}= H^k(x,y,z) &+ \left(-c_{0,0,0} - h_{0,0,0}\right) + \left(-c_{2,1,2} - h_{2,1,2}\right)x^2yz^2 \\ &+ \left(-c_{3,1,3} - h_{3,1,3}\right)x^3yz^3 + \left(-c_{4,1,4} - h_{4,1,4}\right)x^4yz^4
\end{aligned}
\\
\\
&\begin{aligned} = 1+x^2yz^2 + 2x^3yz^3 - x^4yz^4 + \frac{x-2x^2yz^2-3x^3yz^3}{1-x+x^2yz^2+x^3yz^3}
\end{aligned}.
\end{align*}
\end{proof}
\noindent For the following, $S_n$ will represent the traditional star graph on $n$ vertices.
\begin{theorem}\label{theorem37}
If $n+1 \leq N$, then $C^k(K_N^{-S_{n+1}})  = S(N,k) - \underset{j=1}{\overset{n}{\sum}}\binom{n}{j}S(N-j-1, k-1)$.
\end{theorem}

\begin{proof}
Let $s_{j,m,n}$ represent the number of ways of choosing exactly $m$ disjoint components from $S_{n}$ such that the cardinality of the set of vertices of the components chosen is $j$.  It is easy to verify $s_{j,1,n} = \binom{n-1}{j-1}$ for $j > 1$. Noting that every subgraph of $S_n$ is connected yields $m \in \{0,1\}$, which coupled with Remark \ref{rem11} yields
\\
\\
 $s_{j,m,n} = \begin{cases}
 1, & j = m = 0 \\
 \binom{n-1}{j-1}, & m = 1 \text{ and } j > 1\\
 0, & \text{otherwise}
 \end{cases}$
 \\
 \\
 \\
 \\
 so, by Theorem \ref{theorem34},
 \begin{align*}
 C^k(K_N^{-S_{n+1}}) &= \underset{j=0}{\overset{n}{\sum}}\underset{m=0}{\overset{j}{\sum}}(-1)^ms_{j,m,n}S(N-j, k-m)
 \\
 &= S(N,k) - \underset{j=1}{\overset{n}{\sum}}\binom{n}{j}\cdot S(N-j-1, k-1).
 \end{align*}
\end{proof}
\begin{comment}
\begin{corollary}\label{cor6}
If $S^k(x,y) = \underset{j=0}{\overset{\infty}{\sum}}\underset{m=0}{\overset{\infty}{\sum}}\underset{n=0}{\overset{\infty}{\sum}}s_{j,m,n}x^ny^j$, then $S^k(x,y) = \frac{1}{1-x+xy}$.
\end{corollary}

\begin{proof}
The coefficients of formulae in Theorems \ref{theorem7} and \ref{theorem37} match. This means that the generating functions for both sets of coefficients must also match. Hence, the result is true by Theorem \ref{theorem8}.
\end{proof}
\end{comment}
\noindent For the following, $D_n$ will represent the graph of $2n$ vertices with exactly $n$ disjoint edges.
\begin{theorem}\label{theorem38}
If $2n \leq N$, then $C^k(K_N^{-D_{n}})  = S(N,k) - \underset{j=1}{\overset{n}{\sum}}(-1)^j\binom{n}{j}S(N-2j, k-j)$.
\end{theorem}

\begin{proof}
Let $d_{j,m,2n}$ represent the number of ways of choosing exactly $m$ disjoint components from $D_{n}$ such that the cardinality of the set of vertices of the components chosen is $j$. It is easy to verify that
\\
\\
$d_{j,m,2n} = \begin{cases}
          0, & j \not= 2m \\
          \binom{n}{m}, & j = 2m
\end{cases}$
\\
\\
\\
so, by Theorem \ref{theorem34},
\begin{align*}
C(K_N^{-D_n}) &= \underset{j=0}{\overset{2n}{\sum}}\underset{m=0}{\overset{j}{\sum}}(-1)^md_{j,m,2n}S(N-j,k-m)
\\
&= \underset{j=0}{\overset{n}{\sum}}(-1)^jd_{2j,j,2n}S(N-2j,k-j)
\\
&= \underset{j=0}{\overset{n}{\sum}}(-1)^j\binom{n}{j}S(N-2j,k-j).
\end{align*}
\end{proof}

\begin{comment}
\begin{corollary}\label{cor5}
If $D^k(x,y) = \underset{j=0}{\overset{\infty}{\sum}}\underset{m=0}{\overset{\infty}{\sum}}\underset{n=0}{\overset{\infty}{\sum}}d_{j,m,n}x^ny^j$, then $D^k(x,y) = \frac{1}{1-x+xy}$.
\end{corollary}

\begin{proof}
The coefficients of formulae in Theorems \ref{theorem9} and \ref{theorem38} match. This means that the generating functions for both sets of coefficients must also match. Hence, the result is true by Theorem \ref{theorem10}.
\end{proof}
\end{comment}
\bibliographystyle{amsplain}
{\bibliography{GCBib}}

\providecommand{\bysame}{\leavevmode\hbox to3em{\hrulefill}\thinspace}
\providecommand{\MR}{\relax\ifhmode\unskip\space\fi MR }
% \MRhref is called by the amsart/book/proc definition of \MR.
\providecommand{\MRhref}[2]{%
  \href{http://www.ams.org/mathscinet-getitem?mr=#1}{#2}
}
\providecommand{\href}[2]{#2}
\begin{thebibliography}{1}

\bibitem{km}
A.~Knopfmacher and M.E. Mays, \emph{Graph compositions 1: Basic enumeration},
  Integers \textbf{1} (2001), 1--11.

\bibitem{TM}
Todd Tichenor and Michael Mays, \emph{Graph compositions: Deleting edges from
  complete graphs}, Integers \textbf{15} (2015), 1--12.

\end{thebibliography}
\end{document}